\documentclass[12pt]{amsart}

\usepackage{amsmath,amsfonts,amssymb,mathabx,shuffle,latexsym}
\usepackage{enumitem}
\usepackage[usenames, dvipsnames]{xcolor}
\usepackage[backref=page]{hyperref}
\usepackage{ytableau}
\usepackage{tikz-cd}
\usepackage{tikz}
\usetikzlibrary{calc, shapes, backgrounds,arrows,positioning,plotmarks}
\tikzset{>=stealth',
  head/.style = {fill = white, text=black},
  plaque/.style = {draw, rectangle, minimum size = 10mm}, 
  pil/.style={->,thick},
  junct/.style = {draw,circle,inner sep=0.5pt,outer sep=0pt, fill=black}
  }
  
\newtheorem{theorem}{Theorem}[section]
\newtheorem{lemma}[theorem]{Lemma}
\newtheorem{proposition}[theorem]{Proposition}

\newtheorem{conjecture}[theorem]{Conjecture}

\theoremstyle{definition}

\newtheorem{examplex}[theorem]{Example}
\newenvironment{example}
  {\pushQED{\qed}\examplex}
  {\popQED\endexamplex}

\theoremstyle{remark}
\newtheorem{remark}[theorem]{Remark}

\numberwithin{equation}{section}

\newcommand{\wt}{\ensuremath{\mathrm{wt}}}

\newcommand{\inc}{\ensuremath{\mathrm{Inc}}}
\newcommand{\pro}{\mathcal{P}}
\newcommand{\evac}{\mathcal{E}}
\newcommand{\rot}{\ensuremath{\mathsf{rot}}}
\newcommand{\Frame}{\ensuremath{\mathsf{Frame}}}

\newcommand{\LIS}{\ensuremath{\mathsf{LIS}}}
\newcommand{\swap}{\ensuremath{\mathsf{swap}}}
\newcommand{\ein}{\ensuremath{\mathsf{In}}}
\newcommand{\eout}{\ensuremath{\mathsf{Out}}}
\newcommand{\decr}{\ensuremath{\mathsf{Decr}}}

\newcommand{\bulltab}{\ensuremath{\mathsf{BulletTableaux}}}
\newcommand{\slide}{\ensuremath{\mathsf{slide}}}
\newcommand{\rep}{\ensuremath{\mathsf{Rep}}}

\newcommand{\shape}{\ensuremath{\mathsf{sh}}}
\newcommand{\id}{\ensuremath{\mathrm{id}}}
\newcommand{\bbb}{\ensuremath{\mathsf{b}}}
\newcommand{\dist}{\ensuremath{\mathrm{Dist}}}
\newcommand{\rank}{\ensuremath{\mathrm{rank}}}

\begin{document}

%%%%%%%%%%%%%%%%%%%%%%%%%%%%%%%%%%%%%%%%%%%%%%%%%%%%%%%%%%%%
%  TITLE PAGE information
%%%%%%%%%%%%%%%%%%%%%%%%%%%%%%%%%%%%%%%%%%%%%%%%%%%%%%%%%%%%

%     [Short Title]{Full Title}
\title[Promotion of increasing tableaux]{Promotion of increasing tableaux: Frames and homomesies}  

\author[O. Pechenik]{Oliver Pechenik}
\address[OP]{Department of Mathematics, Rutgers University, Piscataway, NJ 08854}
\email{oliver.pechenik@rutgers.edu}
%\thanks{}

%    General info
%\subjclass[2010]{Primary 05E05; Secondary 14M15}

\date{\today}

%\dedicatory{}


\begin{abstract}
A key fact about M.-P.~Sch\"utzenberger's (1972) promotion operator on rectangular standard Young tableaux is that iterating promotion once per entry recovers the original tableau. For tableaux with strictly increasing rows and columns, H.~Thomas and A.~Yong (2009) introduced a theory of $K$-jeu de taquin with applications to $K$-theoretic Schubert calculus. The author (2014) studied a $K$-promotion operator  $\pro$ derived from this theory, but showed that the key fact does not generally extend to $K$-promotion of such increasing tableaux.

Here we show that the key fact holds for labels on the boundary of the rectangle. That is, for $T$ a rectanglar increasing tableau with entries bounded by $q$, we have $\Frame(\pro^q(T)) = \Frame(T)$, where $\Frame(U)$ denotes the restriction of $U$ to its first and last row and column. Using this fact, we obtain a family of homomesy results on the average value of certain statistics over $K$-promotion orbits, extending a $2$-row theorem of J. Bloom, D. Saracino, and the author (2016) to arbitrary rectangular shapes.
\end{abstract}

\maketitle
%\tableofcontents

%%%%%%%%%%%%%%%%%%%%%%%%%%%%%%%%%%%%%%%%%%%%%%%%%%%%%%%%%%%%%%%%
%
\section{Introduction}
%
%%%%%%%%%%%%%%%%%%%%%%%%%%%%%%%%%%%%%%%%%%%%%%%%%%%%%%%%%%%%%%%%
\label{sec:introduction}

An important application of the theory of standard Young tableaux is to the product structure of the cohomology of Grassmannians. Much attention in the modern Schubert calculus has been devoted to the study of analogous problems in $K$-theory (see \cite[\textsection1]{Pechenik.Yong:genomic} for a partial survey of such work). In particular, H.~Thomas and A.~Yong \cite{Thomas.Yong:K} gave a $K$-theoretic Littlewood-Richardson rule by developing a combinatorial theory of \emph{increasing tableaux} as a $K$-theoretic analogue of the classical theory of standard Young tableaux. Their Littlewood-Richardson rule and the associated combinatorics has since been extended to the other \emph{minuscule flag varieties} \cite{Clifford.Thomas.Yong,Buch.Ravikumar,Buch.Samuel} and into \emph{torus-equivariant $K$-theory} \cite{Thomas.Yong:HT,Pechenik.Yong:KT}.

The theory of increasing tableaux is moreover of independent combinatorial interest. Various enumerative combinatorics results have recently been obtained \cite{Pechenik, Pressey.Stokke.Visentin, GMPPRST}; as well as applications to the studies of \emph{combinatorial Hopf algebras} \cite{Patrias.Pylyavskyy}, \emph{longest increasing subsequences of random words} \cite{Thomas.Yong:Plancherel}, \emph{plane partitions} \cite{DPS,HPPW}, and \emph{combinatorial representation theory} \cite{Rhoades:skein}. This paper continues the study begun in \cite{Pechenik} of the $K$-promotion operator on increasing tableaux, a $K$-theoretic analogue of M.-P.~Sch\"utzenberger's \cite{Schutzenberger} classical promotion operator.

We systematically identify a partition $\lambda$ with its Ferrers diagram in English orientation. An {\bf increasing tableau} of shape $\lambda$ is a filling of $\lambda$ by positive integers such that entries strictly increase from left to right across rows and from top to bottom down columns. We write $\inc^q(\lambda)$ for the set of all increasing tableaux of shape $\lambda$ with entries bounded above by $q$.
Using the $K$-theoretic jeu de taquin of \cite{Thomas.Yong:K}, one has a $K$-promotion operator $\pro$ on increasing tableaux \cite{Pechenik} by analogy with M.-P.~Sch\"utzenberger's classical promotion for standard Young tableaux \cite{Schutzenberger}. We describe this operator in detail in Section~\ref{sec:frames}.

The operation of ($K$-)promotion is of particular interest for tableaux of rectangular shapes. For a standard Young tableau $T$ of shape $m\times n$, one has that $\pro^{mn}(T) = T$ (cf.~\cite{Haiman}); indeed, one can completely enumerate the orbits by size in this case via the \emph{cyclic sieving phenomenon} \cite{Rhoades:thesis}. For increasing tableaux, on the other hand, orbits can be much larger than the cardinality $q$ of the alphabet \cite[Example~3.10]{Pechenik}. In general, no upper bound is known on the cardinality of the $K$-promotion orbit of a increasing tableau, even of rectangular shape.

\begin{example}\label{ex:huge_orbit}
Consider the increasing tableau 
\[
T  = \ytableaushort{134578{11}{13}{14}{17},247{10}{12}{13}{15}{17}{19}{21},369{12}{13}{14}{16}{18}{21}{24},68{11}{15}{20}{22}{23}{24}{25}{26}} \in \inc^{26}(4 \times 10).
\]
Although one might naively expect the cardinality of its $K$-promotion orbit to divide $26$ by analogy with the standard Young tableau case, in fact the orbit of $T$ has size $1222 = 26 \cdot 47$.
\end{example}

The {\bf frame} of a partition $\lambda$ is the set $\Frame(\lambda)$ of all boxes in the first or last row, or in the first or last column. For $T \in \inc^q(\lambda)$, we write $\Frame(T)$ for the labeling of $T$ restricted to $\Frame(\lambda)$.

Our first main result is the following:
\begin{theorem}\label{thm:framesame}
Let $T \in \inc^q(m \times n)$. Then \[\Frame(T) = \Frame(\pro^q(T)).\]
\end{theorem}

\begin{example}
Let $T$ be as in Example~\ref{ex:huge_orbit}. Then 
\[
\pro^{26}(T) = \ytableaushort{{\bf 1}{\bf 3}{\bf 4}{\bf 5}{\bf 7}{\bf 8}{\bf 11}{\bf 13}{\bf 14}{\bf 17},{\bf 2}{\bf 4}{6}{7}{10}{12}{14}{15}{\bf 19}{\bf 21},{\bf 3}{\bf 6}{\bf 9}{\bf 12}{\bf 13}{\bf 14}{\bf 16}{\bf 18}{\bf 21}{\bf 24},{\bf 6}{\bf 8}{\bf 11}{\bf 15}{\bf 20}{\bf 22}{\bf 23}{\bf 24}{\bf 25}{\bf 26}},
\] where we have {\bf bolded} all entries that coincide with the corresponding entries of $T$. Note that in accordance with Theorem~\ref{thm:framesame}, every entry of $\Frame(\pro^{26}(T))$ is bold.
\end{example}

\begin{remark}
Since $\Frame(2 \times n) = 2 \times n$, Theorem~\ref{thm:framesame} in particular recovers the author's previous result \cite[Theorem~1.3]{Pechenik} that $\pro^q(T) = T$ for $T \in \inc^q(2 \times n)$.
\end{remark}

The following was conjectured in work with K.~Dilks and J.~Striker \cite[Conjecture~4.12]{DPS}:

\begin{conjecture}\label{conj:3row}
Let $T \in \inc^q(3 \times n)$. Then $T = \pro^q(T)$.
\end{conjecture}
Theorem~\ref{thm:framesame} may be interpreted as evidence toward Conjecture~\ref{conj:3row}, since Theorem~\ref{thm:framesame} shows that $T$ and $\pro^q(T)$ have the same entries in at least $2n+2$ out of $3n$ pairs of corresponding boxes. 

%We will use this observation in Theorem~\ref{thm:order_bound} to obtain the first nontrivial bound on the order of $K$-promotion for general rectangular shapes.

A set $U$ of objects with a weight function $\wt : U \to \mathbb{C}$ and a group action $G \curvearrowright U$ is said to be {\bf homomesic} if every $G$-orbit $\mathcal{O}$ has the same average weight $\frac{\sum_{x \in \mathcal{O}} \wt(x)}{|\mathcal{O}|}$. This notion was isolated by J.~Propp and T.~Roby \cite{Propp.Roby} in response to observations of D.~Panyushev \cite{Panyushev}, and has since been found in to appear in diverse situations \cite{Einstein.Propp, Hopkins.Zhang,Striker,Rush.Wang,Dong.Wang,EFGJMPR,Joseph.Roby}.

Using Theorem~\ref{thm:framesame}, we obtain our second main result, a family of new homomesies for increasing tableaux. For $T \in \inc^q(\lambda)$ and $S$ a set of boxes in $\lambda$, let $\wt_S(T)$ denote the sum of the entries of $T$ in $S$.
\begin{theorem}\label{thm:homomesy}
Let $S$ be subset of $\Frame(m \times n)$ that is fixed under $180^\degree$ rotation. 
Then $(\inc^q(m \times n), \pro, \wt_S)$ exhibits homomesy with orbit average $\frac{(q+1)|S|}{2}$.
\end{theorem}
%(I had suggested Theorem~\ref{thm:homomesy} as something possibly true in a semiprivate email to the DAC community before the listserve was created.)

\begin{remark}
The case $m=2$ of Theorem~\ref{thm:homomesy} was previously proved by J.~Bloom, O.~Pechenik and D.~Saracino \cite[Theorem~6.5]{BPS} using results from \cite{Pechenik}. 

The analogue of Theorem~\ref{thm:homomesy} for (semi)standard Young tableaux was conjectured by J.~Propp and T.~Roby \cite{Propp.Roby:dartmouth} and proved by J.~Bloom, O.~Pechenik and D.~Saracino \cite[Theorem~1.1]{BPS}. In fact, for (semi)standard Young tableaux, $S$ need not be contained in $\Frame(m \times n)$. However, \cite[Example~6.4]{BPS} shows that for increasing tableaux a generalization of Theorem~\ref{thm:homomesy} without the condition $S \subseteq \Frame(m \times n)$ would be false.
\end{remark}

\section{$K$-jeu de taquin and frames of increasing tableaux}\label{sec:frames}
The section culminates in a proof of Theorem~\ref{thm:framesame}. First we recall the $K$-jeu de taquin of H.~Thomas and A.~Yong \cite{Thomas.Yong:K}, the key ingredient in the operation of $K$-promotion on increasing tableaux. While $K$-promotion can be defined without a full development of $K$-jeu de taquin, we will need $K$-jeu de taquin in the proof of Theorem~\ref{thm:framesame}.

\subsection{$K$-jeu de taquin}
Let $\bulltab(\nu / \lambda)$ denote the set of all fillings of the skew shape $\nu / \lambda$ by positive integers and symbols $\bullet$.
For each positive integer $i$, we define as follows an operator $\swap_i$ on $\bulltab(\nu / \lambda)$. Let $T \in \bulltab(\nu / \lambda)$ and consider the boxes of $T$ that contain either $i$ or $\bullet$. The set of such boxes decomposes into edge-connected components. On each such component that is a single box, $\swap_i$ does nothing. On each nontrivial component, $\swap_i$ simultaneously replaces each $i$ by $\bullet$ and each $\bullet$ by $i$. The resulting element of $\bulltab(\nu / \lambda)$ is $\swap_i(T)$. 

\begin{example}
\pushQED{\qed}
Consider 
\[ 
T = \ytableaushort{4 7 3 \bullet 2 2,1 2 \bullet 2, \bullet 3} \in \bulltab \left( \ytableausetup{smalltableaux} \raisebox{3mm}{\ydiagram{6,4,2}} \right). \ytableausetup{nosmalltableaux}
\]
In computing $\swap_2(T)$, one looks at two connected components. The southwest component is a single box containing $\bullet$ and is unchanged by $\swap_2$. The other component has six boxes. Hence 
\[
\swap_2(T) = \ytableaushort{4 7 3 2 \bullet \bullet, 1 \bullet 2 \bullet, \bullet 3}. \qedhere \popQED
\] \let\qed\relax
%Also 
%\[
%\ein_7(T) = \ytableaushort{4 7 3 7 2 2,1 2 7 2, 7 3} \text{ and } \eout_7(T) = \ytableaushort{4 \blank 3 \blank 2 2,1 2 \blank 2, \blank 3}.
%\]
\end{example}

For a box $\bbb$ in a partition, we write $\bbb^\rightarrow$ for the box immediately right of $\bbb$ in its row, $\bbb^\downarrow$ for the box immediately below $\bbb$ in its column, etc.

Consider a skew shape $\nu / \lambda$. An {\bf inner corner} of $\nu/\lambda$ is a box $\bbb \in \lambda$ such that $\bbb^\rightarrow \notin \lambda$ and $\bbb^\downarrow \notin \lambda$. For $I$ any nonempty set of inner corners of $\nu / \lambda$ and $T \in \inc^q(\nu / \lambda)$, let $\ein_I(T)$ be the extension of $T$ formed by adding a $\bullet$ to each box of $I$. Note that $\ein_I(T) \in \bulltab(\nu / \theta)$ for some $\theta \supset \lambda$.

An {\bf outer corner} of $\nu / \lambda$ is a box $\bbb \in \nu / \lambda$ such that $\bbb^\rightarrow \notin \nu / \lambda$ and $\bbb^\downarrow \notin \nu / \lambda$. If $T \in \bulltab(\nu / \lambda)$ has all $\bullet$'s in outer corners, then we define $\eout^\bullet(T)$ to be the filling obtained by deleting every $\bullet$ from $T$; otherwise $\eout^\bullet(T)$ is undefined. Note that if $\eout^\bullet(T)$ is defined, then it has shape $\delta / \lambda$ for some $\delta \subseteq \nu$.

Let $T \in \inc^q(\nu / \lambda)$ and let $I$ be any nonempty set of inner corners of $\nu / \lambda$. Then the {\bf $K$-jeu de taquin slide} of $T$ at $I$ is the result of the following composition of operations
\[
\slide_I(T) := \eout^\bullet \circ \swap_q \circ \dots \circ \swap_2 \circ \swap_1 \circ \ein_I(T).
\]
Observe that $\slide_I(T) \in \inc^q(\delta/ \rho)$ for some $\rho \subset \lambda$ and $\delta \subset \nu$.

Iterating this process for successive nonempty sets of inner corners $I_1, I_2, \ldots$, one eventually obtains an increasing tableau $R \in \inc^q(\kappa)$ of some straight shape $\kappa$. Such a tableau is called a {\bf rectification} of $T$.

\begin{remark}
Unlike in the classical standard tableau setting, an increasing tableau $T \in \inc^q(\nu / \lambda)$ may have more than one rectification and these rectifications may moreover have different straight shapes.
\end{remark}

\subsection{$K$-promotion}

For $T \in \bulltab(\nu/\lambda)$, we define an operation $\rep_{1 \rightarrow \bullet}$ that replaces each instance of $1$ by $\bullet$, as well as, for each $n \in \mathbb{Z}_{> 0}$, an operation $\rep_{\bullet \rightarrow n}$ that replaces each instance of $\bullet$ by $n$.
Let $\decr$ be the operator that decrements each numerical entry by $1$ (and ignores $\bullet$'s).

{\bf $K$-promotion} on $\inc^q(\lambda) \subset \bulltab(\lambda)$ is the composition \[\pro := \decr \circ \rep_{\bullet\rightarrow q+1} \circ \swap_{q} \circ \cdots \circ \swap_3 \circ \swap_2  \circ \rep_{1 \rightarrow \bullet}.\] It is not hard to see that if $T \in \inc^q(\lambda)$, then $\pro(T) \in \inc^q(\lambda)$, and that moreover this operation coincides with M.-P.~Sch\"utzenberger's definition of promotion \cite{Schutzenberger} in the case that $T$ is a standard Young tableau. For more details, see \cite{Pechenik}.

\begin{example}
Let $T  = \ytableaushort{124,346} \in \inc^6(2 \times 3)$. Then one computes $\pro(T)$ as follows:
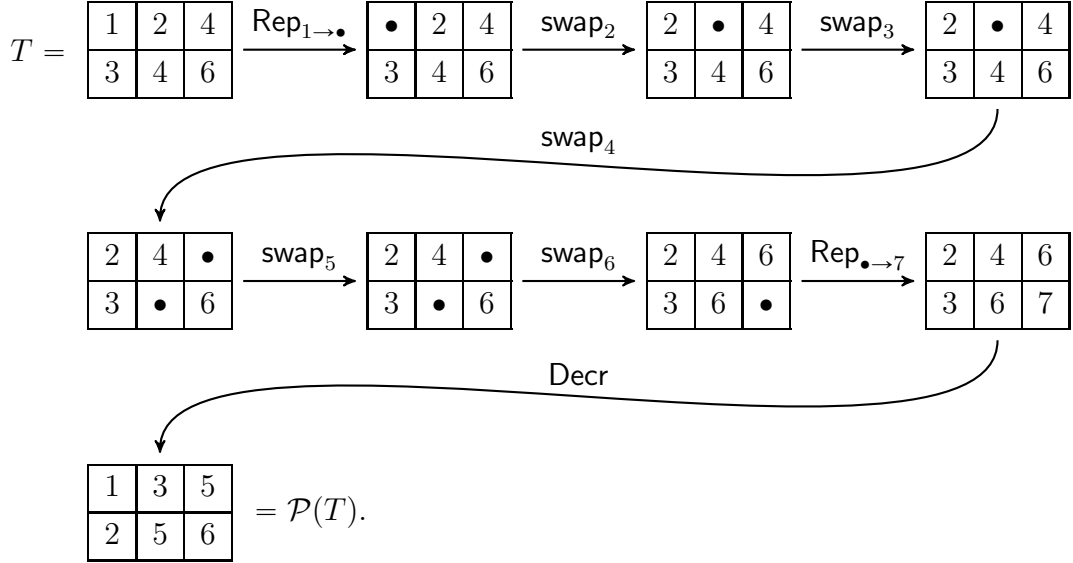
\[
\begin{tikzpicture}
\node (A) {\ytableaushort{124,346}};
\node[left = 0 of A] (T) {$T=$};
\node[right=1.5 of A] (B) {\ytableaushort{\bullet 24, 346}};
\node[right=1.5 of B] (C) {\ytableaushort{2 \bullet 4, 346}};
\node[right=1.5 of C] (D) {\ytableaushort{2 \bullet 4, 346}};
\node[below=1.5 of A] (E) {\ytableaushort{24 \bullet, 3 \bullet 6}};
\node[right=1.5 of E] (F) {\ytableaushort{24 \bullet, 3 \bullet 6}};
\node[right=1.5 of F] (G) {\ytableaushort{24 6, 3 6 \bullet}};
\node[right=1.5 of G] (H) {\ytableaushort{24 6, 3 6 7}};
\node[below=1.5 of E] (I) {\ytableaushort{135, 256}};
\node[right = 0 of I] (J) {$=\pro(T).$};
\path (A) edge[pil] node[above]{$\rep_{1\rightarrow \bullet}$} (B);
\path (B) edge[pil] node[above]{$\swap_2$} (C);
\path (C) edge[pil] node[above]{$\swap_3$} (D);
\draw[->, thick] (D.south) .. node[above]{$\swap_4$} controls ([yshift=-3cm] D) and ([yshift=3cm] E) .. (E.north);
\path (E) edge[pil] node[above]{$\swap_5$} (F);
\path (F) edge[pil] node[above]{$\swap_6$} (G);
\path (G) edge[pil] node[above]{$\rep_{\bullet \rightarrow 7}$} (H);
\draw[->, thick] (H.south) .. node[above]{$\decr$} controls ([yshift=-3cm] H) and ([yshift=3cm] I) .. (I.north);
\end{tikzpicture}
\]
\end{example}

\subsection{$K$-evacuation and its dual}
To prove Theorem~\ref{thm:framesame} on $K$-promotion, we will need the related notion of (dual) $K$-evacuation. Define the {\bf shape} of an increasing tableau $T$ to be $\shape(T) = \lambda$ if $T \in \inc^q(\lambda)$ for some $q$. Write $T_{\leq a}$ for the subtableau of $T$ given by deleting all entries greater than $a$ and removing all empty boxes. In analogous fashion, define $T_{<a}, T_{\geq a},$ and $T_{>a}$, where $T_{\geq a}$ and $T_{>a}$ will generally be of skew shape. Note that $T \in \inc^q(\lambda)$ is uniquely determined by the vector of partitions
\[
\Big(\shape(T_{\leq 0}), \shape(T_{\leq 1}), \dots, \shape(T_{\leq q}) \Big).
\]

For $T \in \inc^q(\lambda)$, we define the {\bf $K$-evacuation} of $T$ to be the tableau $\evac(T)$ encoded by the vector
\[
\Big(\shape(\pro^q(T)_{\leq 0}), \shape(\pro^{q-1}(T)_{\leq 1}), \dots, \shape(\pro^0(T)_{\leq q}) \Big).
\]
Similarly, the {\bf dual $K$-evacuation} of $T$ is $\evac^*(T)$ encoded by the vector
\[
\Big(\shape(\pro^0(T)_{\leq 0}), \shape(\pro^{-1}(T)_{\leq 1}), \dots, \shape(\pro^{-q}(T)_{\leq q}) \Big).
\]

It is useful to encode all these data in a {\bf $K$-theoretic growth diagram} as in \cite{Thomas.Yong:K}, using ideas that originate in work of S.~Fomin (cf.\ \cite[Appendix~1]{EC2}); the $K$-theoretic growth diagram for $T \in \inc^q(\lambda)$ is a semi-infinite $2$-dimensional array formed by placing the partition $\shape(\pro^j(T)_{\leq i})$ in position $(i +j,j)$, where $0 \leq i \leq q$ and $j \in \mathbb{Z}$.

\begin{example}\label{ex:growth_diagram}
Let \[
T = \ytableaushort{1245,3458,4679,68{10}{11}}.
\]
Then the $K$-theoretic growth diagram of $T$ is 
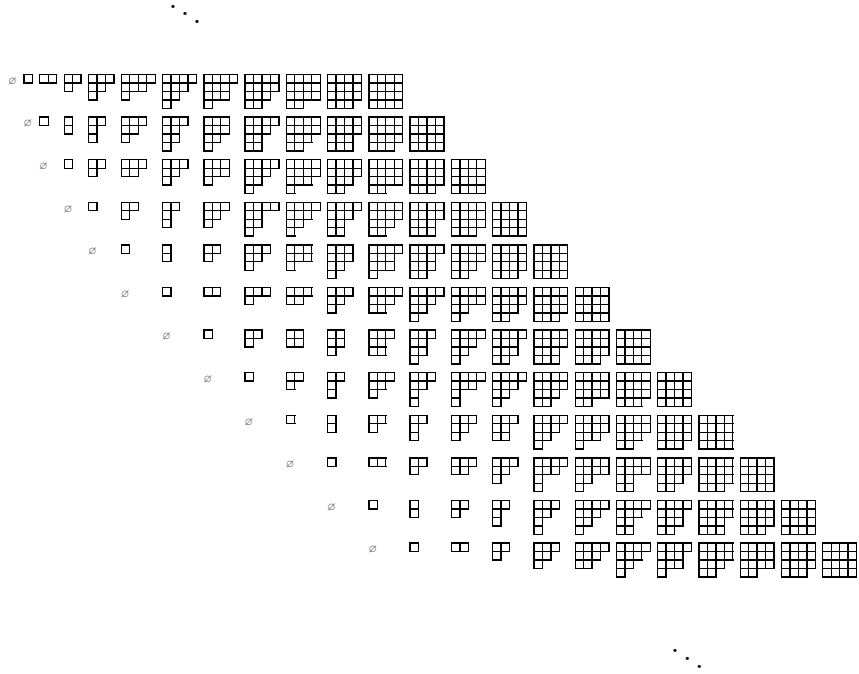
\[ \ddots \phantom{35246427572435135234544523456222} \]
\[\ytableausetup{boxsize=4mm}
\scalebox{0.27}{$\begin{array}{llllllllllllllllllllllll}
\emptyset & \ydiagram{1} & \ydiagram{2} & \ydiagram{2,1} & \ydiagram{3,2,1} & \ydiagram{4,3,1} & \ydiagram{4,3,2,1} & \ydiagram{4,3,3,1} & \ydiagram{4,4,3,2} & \ydiagram{4,4,4,2} & \ydiagram{4,4,4,3} & \ydiagram{4,4,4,4,0} & & & & &  &  & & &  &  & & \\ %homerow
 & \emptyset & \ydiagram{1} & \ydiagram{1,1} & \ydiagram{2,1,1} & \ydiagram{3,2,1} & \ydiagram{3,2,2,1} &\ydiagram{3,3,2,1} & \ydiagram{4,3,2,2} & \ydiagram{4,4,3,2} & \ydiagram{4,4,3,3} & \ydiagram{4,4,4,3} & \ydiagram{4,4,4,4,0} & & & &  &  & &&  &  & &\\ %P1
 & & \emptyset & \ydiagram{1} & \ydiagram{2,1} & \ydiagram{3,2} & \ydiagram{3,2,1} & \ydiagram{3,3,1} & \ydiagram{4,3,2,1} & \ydiagram{4,4,3,1} & \ydiagram{4,4,3,2} & \ydiagram{4,4,4,2} & \ydiagram{4,4,4,3} & \ydiagram{4,4,4,4,0} & & &  &  & &&  &  & &\\ %P2
& & & \emptyset & \ydiagram{1} & \ydiagram{2,1} & \ydiagram{2,1,1} & \ydiagram{3,2,1} & \ydiagram{4,2,2,1} & \ydiagram{4,3,2,1} & \ydiagram{4,3,2,2} & \ydiagram{4,4,3,2} & \ydiagram{4,4,3,3} & \ydiagram{4,4,4,3} & \ydiagram{4,4,4,4,0} & &  &  & & &  &  & &\\ %P3
& & &  & \emptyset  & \ydiagram{1}  & \ydiagram{1,1}  & \ydiagram{2,1}  & \ydiagram{3,2,1}  & \ydiagram{3,3,1}  & \ydiagram{3,3,2,1}  & \ydiagram{4,3,3,1}  & \ydiagram{4,3,3,2}  & \ydiagram{4,4,3,2}  & \ydiagram{4,4,4,3} & \ydiagram{4,4,4,4,0} &  &  & & &  &  & & \\ %P4
& & &  &  & \emptyset  & \ydiagram{1}  & \ydiagram{2}  & \ydiagram{3,1}  & \ydiagram{3,2}  & \ydiagram{3,2,1}  & \ydiagram{4,3,2}  & \ydiagram{4,3,2,1}  & \ydiagram{4,4,2,1}  & \ydiagram{4,4,3,2} & \ydiagram{4,4,4,3} &\ydiagram{4,4,4,4,0} &  & & &  &  & &\\ %P5
& & &  &  &  & \emptyset  & \ydiagram{1}  & \ydiagram{2,1}  & \ydiagram{2,2}  & \ydiagram{2,2,1}  & \ydiagram{3,2,2}  & \ydiagram{3,2,2,1}  & \ydiagram{4,3,2,1}  &\ydiagram{4,3,3,2} & \ydiagram{4,4,3,3} & \ydiagram{4,4,4,3}  & \ydiagram{4,4,4,4,0}  & & &  &  & &\\ %P6
& & &  &  &  &  & \emptyset  & \ydiagram{1}  & \ydiagram{2,1} & \ydiagram{2,1,1}  & \ydiagram{3,2,1}  & \ydiagram{3,2,1,1}  & \ydiagram{4,3,1,1}  & \ydiagram{4,3,2,1} & \ydiagram{4,4,3,2} & \ydiagram{4,4,4,2}  & \ydiagram{4,4,4,3}  & \ydiagram{4,4,4,4,0} & &  &  & &\\ %P7
& & &  &  &  &  &  & \emptyset & \ydiagram{1} & \ydiagram{1,1}  &\ydiagram{2,1}  & \ydiagram{2,1,1}  & \ydiagram{3,2,1}  & \ydiagram{3,2,2} & \ydiagram{4,3,2,1} & \ydiagram{4,4,3,1}  & \ydiagram{4,4,3,2}  & \ydiagram{4,4,4,3} & \ydiagram{4,4,4,4,0} &  &  & &\\ %P8
& & &  &  &  &  &  &  & \emptyset  & \ydiagram{1} &  \ydiagram{2} &  \ydiagram{2,1} & \ydiagram{3,2}  & \ydiagram{3,2,1} & \ydiagram{4,3,1,1} & \ydiagram{4,4,2,1}  & \ydiagram{4,4,2,2}  & \ydiagram{4,4,3,2} & \ydiagram{4,4,4,3} & \ydiagram{4,4,4,4,0}  &  & &\\ %P9
& & &  &  &  &  &  &  &  & \emptyset  & \ydiagram{1}  & \ydiagram{1,1}  & \ydiagram{2,1}  & \ydiagram{2,1,1} & \ydiagram{3,2,1,1} & \ydiagram{4,3,2,1}  & \ydiagram{4,3,2,2}  & \ydiagram{4,3,3,2} & \ydiagram{4,4,3,3} &  \ydiagram{4,4,4,3} & \ydiagram{4,4,4,4,0}  & &\\ %P10
& & &  &  &  &  &  &  &  &  & \emptyset  & \ydiagram{1}  & \ydiagram{2}  & \ydiagram{2,1} & \ydiagram{3,2,1} & \ydiagram{4,3,2}  & \ydiagram{4,3,2,1}  & \ydiagram{4,3,3,1} & \ydiagram{4,4,3,2} & \ydiagram{4,4,4,2}  & \ydiagram{4,4,4,3}  & \ydiagram{4,4,4,4,0} &\\ %P11
%& & &  &  &  &  &  &  &  &  &  &  &  & & &  &  & & &  &  & &\\
%& & &  &  &  &  &  &  &  &  &  &  &  & & &  &  & & &  &  & &\\
%& & &  &  &  &  &  &  &  &  &  &  &  & & &  &  & & &  &  & &\\
\end{array}$}
\] \ytableausetup{boxsize=normal}
\[ \phantom{35246427572435135234544523456222} \ddots \]

Here the top illustrated row encodes $T$, the bottom row encodes $\pro^{11}(T)$, and the central column encodes $\evac(T)$ (which is also $\evac^*(\pro^{11}(T))$).
\end{example}

Using the $K$-theoretic growth diagram, it is not hard to uncover various relations between the operators under consideration. Together \cite[Theorem~4.1]{Thomas.Yong:K} and \cite[Lemma~3.1]{Pechenik} give the following facts that we will need:
\begin{lemma}\label{lemma:GD_facts}
The following relations hold for operations on $\inc^q(\lambda)$:
\begin{itemize}
\item[(a)] $\evac^2 = (\evac^*)^2 = \id$;
\item[(b)] $\evac^* \circ \evac = \pro^q$;
\item[(c)] $\pro \circ \evac = \evac \circ \pro^{-1}$;
\item[(d)] if $\lambda = m \times n$, $\evac^* = \rot \circ \evac \circ \rot$, where $\rot(T)$ is given by rotating $T$ by $180^\degree$ and replacing $i$ by $q+1-i$.
\end{itemize}
\end{lemma}

\subsection{Proof of Theorem~\ref{thm:framesame}}
Let $w$ be the {\bf reading word} of $T$, given by reading the entries of $T$ by rows from left to right and from bottom to top, i.e. in ``reverse Arabic fashion.'' Let $\rot(w) := w_0 \cdot w \cdot w_0$, where $w_0$ is the longest element of the symmetric group $S_q$. Since $\rot(w)$ is obtained from $w$ by reversing the order of the letters of $w$ and then replacing $i$ by $q+1-i$, we see that $\rot(w)$ is the reading word of $\rot(T)$.

Define $w_{\leq a}$ to be the subword of $w$ obtained by deleting all letters greater than $a$, with analogous definitions of $w_{< a}$,  $w_{\geq a}$, and  $w_{> a}$.

\begin{lemma}\label{lem:first_rows}
The tableaux $\rot(T)$ and $\evac(T)$ have the same first row. 
\end{lemma}
\begin{proof}
 The reading word of  $\rot(T)_{\leq a}$ is $\rot(w)_{\leq a}$. Hence by \cite[Theorem~6.1]{Thomas.Yong:K}, the length of the first row of $\rot(T)_{\leq a}$ is $\LIS(\rot(w)_{\leq a})$, where $\LIS(u)$ denotes the length of the longest strictly increasing subsequence of the word $u$.  By the definition of $\rot$, we have $\LIS(\rot(w)_{\leq a}) = \LIS(w_{> n-a})$. But by \cite[Theorem~6.1]{Thomas.Yong:K}, $\LIS(w_{> n-a})$ is the length of the first row of any $K$-rectification of $T_{> n-a}$.
 By definition, the shape of $\evac(T)_{\leq a}$ is the shape of a particular $K$-rectification of $T_{> n-a}$. Thus the length of the first row of $\evac(T)_{\leq a}$ is also the length of the first row of $\rot(T)_{\leq a}$. The lemma follows.
\end{proof}

\begin{lemma}\label{lem:first_columns}
The tableaux $\rot(T)$ and $\evac(T)$ have the same first column. 
\end{lemma}
\begin{proof}
The proof is the same as for Lemma~\ref{lem:first_rows}, except that one should replace use of \cite[Theorem~6.1]{Thomas.Yong:K} on the relation between first rows and longest increasing subsequences with use of the analogous relation between first \emph{columns} and longest \emph{decreasing} subsequences (see \cite{Thomas.Yong:Plancherel} or \cite[Corollary~6.8]{Buch.Samuel}).
\end{proof}

The following proposition is of independent interest. It extends \cite[Proposition~3.3]{Pechenik}, which is the special case where $T \in \inc^q(2 \times n)$.

\begin{proposition}\label{prop:frame}
The tableaux $\rot(T)$ and $\evac(T)$ have the same frame. 
\end{proposition}
\begin{proof}
By Lemmas~\ref{lem:first_rows} and~\ref{lem:first_columns}, it remains to show that $\rot(T)$ and $\evac(T)$ have the same last row and column.

Let $T' = \evac(T)$. Then by Lemmas~\ref{lem:first_rows} and~\ref{lem:first_columns}, $\rot(T')$ and $\evac(T')$ have the same first row and column. But by Lemma~\ref{lemma:GD_facts}(a), $\evac(T') = T$, so $\rot(T')$ and $T$ have the same first row and column. Hence $\rot(\rot(T'))$ and and $\rot(T)$ have the same last row and column. Since $\rot(\rot(T')) = \evac(T)$, we are done.
\end{proof}

Clearly $\rot(T)$ and $\rot(\rot(\rot(T)))$ have the same frame. Since by Lemma~\ref{lemma:GD_facts}(d), we have $\evac^* = \rot \circ \evac \circ \rot$, it thereby follows from Proposition~\ref{prop:frame} that $\rot(T)$ and $\evac^*(T)$ also have the same frame. Thus $\evac^* ( \evac(T))$ has the same frame as $\rot(\rot(T)) = T$. But by Lemma~\ref{lemma:GD_facts}(b), $\pro^q(T) = \evac^* ( \evac(T))$, so \[\Frame(T) = \Frame(\pro^q(T)).\] This concludes the proof of Theorem~\ref{thm:framesame}. \qed
%Since $T$ and $\rot(T)$ are URTs, $T$ is the unique tableau associated to $w$ and $\rot(w)$ is the unique tableau associated to $\rot(w)$.

\section{Homomesy}\label{sec:homomesy}
In this section, we prove a family of new homomesy results for increasing tableaux. We will obtain these by imitating the proof of \cite[Theorem~1.1]{BPS} and using Theorem~\ref{thm:framesame}.

For $T \in \inc^q(m \times n)$ and $\bbb$ a box in $\Frame(m \times n)$, let $\dist(T, \bbb)$ be the \emph{multiset} \[
\dist(T, \bbb)  := \{  \wt_{\{ \bbb \}}(\pro^k(T)) : 0 \leq k < q   \}.
\]

\begin{proposition}\label{prop:dist}
For $T \in \inc^q(m \times n)$ and $\bbb$ a box in $\Frame(m \times n)$, \[\dist(T, \bbb) = \dist(\evac(T), \bbb).\]
\end{proposition}
\begin{proof}
This proof is perhaps best understood by following along with the succeeding Example~\ref{ex:shading}.

Consider the $K$-theoretic growth diagram $G$ for $T$. A fixed row $r$ of $G$ encodes an increasing tableau $R$. The row immediately below this encodes $\pro(R)$. The column that intersects $r$ at its rightmost partition encodes, by definition, $\evac(R)$. The column immediately left of this then encodes $\pro(\evac(R))$ by Lemma~\ref{lemma:GD_facts}(c). Say the {\bf rank} of a partition $\pi$ in $G$ is the number $\rank(\pi)$ of partitions that are strictly left of $\pi$ and in its row. Note that the rank is also the number of partitions strictly below $\pi$ in its column.

Shade each partition in $G$ that contains the box $\bbb$. For any set of $q$ consecutive rows $\{ r_i : 0 < i \leq q \}$, we have by Theorem~\ref{thm:framesame} the equality of multisets
\[
\dist(T, \bbb) = \{ \rank(\rho_i)  : 0 < i \leq q \},
\]
where $\rho_i$ is the leftmost shaded partition in row $r_i$. 
In the same way, for any set of $q$ consecutive columns $\{ c_j : 0 \leq j < q \}$, we have 
\[ 
\dist(\evac(T), \bbb) = \{ \rank(\gamma_j)  : 0 \leq j < q \},
\]
where $\gamma_j$ is the bottommost shaded partition in column $c_j$.

Fix $C \in \mathbb{Z}$. For $1 \leq k \leq 2q + 1$, let $d_k$ denote the diagonal line of slope one through $G$ given by $y = x - k + C$. 
For each diagonal $d_k$, let $\delta_k$ be the smallest shaded partition that lies on $d_k$.
Observe that the partitions $\delta_k$ are restricted to $q+1$ consecutive rows $\{ r_i : 0 \leq i \leq q \}$ of $G$ and to $q+1$ consecutive columns $\{ c_j : 0 \leq j \leq q \}$ of $G$. We have the equalities of multisets 
\[
\{ \rank(\rho_i)  : 0 < i \leq q \} = \{ \rank(\delta_k) :  1 < k \leq 2q+1, \rank(\delta_k) = \rank(\delta_{k-1}) -1 \}
\]
and
\[
\{ \rank(\gamma_j)  : 0 \leq j < q \} = \{ \rank(\delta_k) : 1 \leq k < 2q+1, \rank(\delta_k)  = \rank(\delta_{k+1}) - 1 \}.
\]

Now construct a lattice path $P$ in the first quadrant of the plane by plotting the points $(k, \rank(\delta_k))$ for $1 \leq k \leq 2q + 1$ and connecting the vertex $(k, \rank(\delta_k))$ to the vertex $(k+1, \rank(\delta_{k+1}))$ by a line segment. Note that \[
\rank(\delta_1) = \rank(\delta_{2q+1})
\]
 by Theorem~\ref{thm:framesame}.
Moreover, \[\rank(\delta_{k+1}) = \rank(\delta_k) \pm 1\] for all $k$. Hence for any positive integer $h$, the number of $k$ in the interval $[2, 2q+1]$  with $\rank(\delta_k) = h = \rank(\delta_{k-1}) - 1$ equals the number of $k$ in the interval $[1, 2q]$ with  $\rank(\delta_k) = h = \rank(\delta_{k+1}) - 1$.
%Now $\delta_k$ is the leftmost shaded partition in its row exactly if $\rank(\delta_k)) = \rank(\delta_{k-1}) - 1$. Moreover $\delta_k$ is the bottommost shaded partition in its column exactly if $\rank(\delta_k)) = \rank(\delta_{k-1}) + 1$. 
This proves that 
\[
\{ \rank(\rho_i)  : 0 < i \leq q \} = \{ \rank(\gamma_j)  : 0 \leq j < q \}
\] and the proposition follows.
\end{proof}

\begin{example}\label{ex:shading}
Extending Example~\ref{ex:growth_diagram}, let $\bbb$ be the shaded box in \[
T = \ytableaushort{1245,345{*(SkyBlue) 8},4679,68{10}{11}}.
\]
Then we shade the $K$-theoretic growth diagram of $T$ as 
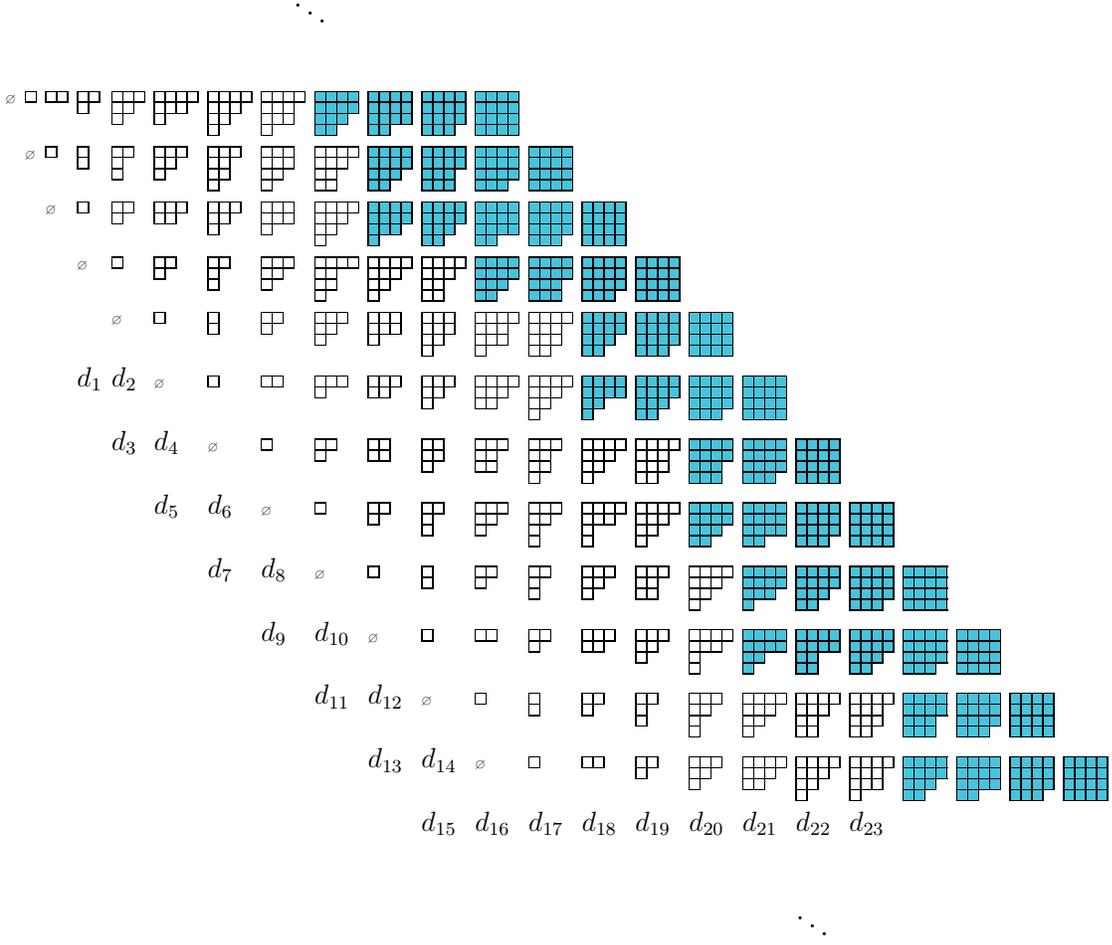
\[ \ddots \phantom{35246427572435135234544523456222} \]
\[\ytableausetup{boxsize=4mm}
\scalebox{0.35}{$\begin{array}{llllllllllllllllllllllll}
\emptyset & \ydiagram{1} & \ydiagram{2} & \ydiagram{2,1} & \ydiagram{3,2,1} & \ydiagram{4,3,1} & \ydiagram{4,3,2,1} & \ydiagram{4,3,3,1} & \ydiagram[*(SkyBlue)]{4,4,3,2} & \ydiagram[*(SkyBlue)]{4,4,4,2} & \ydiagram[*(SkyBlue)]{4,4,4,3} & \ydiagram[*(SkyBlue)]{4,4,4,4,0} & & & & &  &  & & &  &  & & \\ %homerow
 & \emptyset & \ydiagram{1} & \ydiagram{1,1} & \ydiagram{2,1,1} & \ydiagram{3,2,1} & \ydiagram{3,2,2,1} &\ydiagram{3,3,2,1} & \ydiagram{4,3,2,2} & \ydiagram[*(SkyBlue)]{4,4,3,2} & \ydiagram[*(SkyBlue)]{4,4,3,3} & \ydiagram[*(SkyBlue)]{4,4,4,3} & \ydiagram[*(SkyBlue)]{4,4,4,4,0} & & & &  &  & &&  &  & &\\ %P1
 & & \emptyset & \ydiagram{1} & \ydiagram{2,1} & \ydiagram{3,2} & \ydiagram{3,2,1} & \ydiagram{3,3,1} & \ydiagram{4,3,2,1} & \ydiagram[*(SkyBlue)]{4,4,3,1} & \ydiagram[*(SkyBlue)]{4,4,3,2} & \ydiagram[*(SkyBlue)]{4,4,4,2} & \ydiagram[*(SkyBlue)]{4,4,4,3} & \ydiagram[*(SkyBlue)]{4,4,4,4,0} & & &  &  & &&  &  & &\\ %P2
& & & \emptyset & \ydiagram{1} & \ydiagram{2,1} & \ydiagram{2,1,1} & \ydiagram{3,2,1} & \ydiagram{4,2,2,1} & \ydiagram{4,3,2,1} & \ydiagram{4,3,2,2} & \ydiagram[*(SkyBlue)]{4,4,3,2} & \ydiagram[*(SkyBlue)]{4,4,3,3} & \ydiagram[*(SkyBlue)]{4,4,4,3} & \ydiagram[*(SkyBlue)]{4,4,4,4,0} & &  &  & & &  &  & &\\ %P3
& & &  & \emptyset  & \ydiagram{1}  & \ydiagram{1,1}  & \ydiagram{2,1}  & \ydiagram{3,2,1}  & \ydiagram{3,3,1}  & \ydiagram{3,3,2,1}  & \ydiagram{4,3,3,1}  & \ydiagram{4,3,3,2}  & \ydiagram[*(SkyBlue)]{4,4,3,2}  & \ydiagram[*(SkyBlue)]{4,4,4,3} & \ydiagram[*(SkyBlue)]{4,4,4,4,0} &  &  & & &  &  & & \\ %P4
& & & \scalebox{2.5}{$d_1$} & \scalebox{2.5}{$d_2$}  & \emptyset  & \ydiagram{1}  & \ydiagram{2}  & \ydiagram{3,1}  & \ydiagram{3,2}  & \ydiagram{3,2,1}  & \ydiagram{4,3,2}  & \ydiagram{4,3,2,1}  & \ydiagram[*(SkyBlue)]{4,4,2,1}  & \ydiagram[*(SkyBlue)]{4,4,3,2} & \ydiagram[*(SkyBlue)]{4,4,4,3} &\ydiagram[*(SkyBlue)]{4,4,4,4,0} &  & & &  &  & &\\ %P5
& & &  & \scalebox{2.5}{$d_3$}  & \scalebox{2.5}{$d_4$}  & \emptyset  & \ydiagram{1}  & \ydiagram{2,1}  & \ydiagram{2,2}  & \ydiagram{2,2,1}  & \ydiagram{3,2,2}  & \ydiagram{3,2,2,1}  & \ydiagram{4,3,2,1}  &\ydiagram{4,3,3,2} & \ydiagram[*(SkyBlue)]{4,4,3,3} & \ydiagram[*(SkyBlue)]{4,4,4,3}  & \ydiagram[*(SkyBlue)]{4,4,4,4,0}  & & &  &  & &\\ %P6
& & &  &  & \scalebox{2.5}{$d_5$}  & \scalebox{2.5}{$d_6$}  & \emptyset  & \ydiagram{1}  & \ydiagram{2,1} & \ydiagram{2,1,1}  & \ydiagram{3,2,1}  & \ydiagram{3,2,1,1}  & \ydiagram{4,3,1,1}  & \ydiagram{4,3,2,1} & \ydiagram[*(SkyBlue)]{4,4,3,2} & \ydiagram[*(SkyBlue)]{4,4,4,2}  & \ydiagram[*(SkyBlue)]{4,4,4,3}  & \ydiagram[*(SkyBlue)]{4,4,4,4,0} & &  &  & &\\ %P7
& & &  &  &  & \scalebox{2.5}{$d_7$}  & \scalebox{2.5}{$d_8$}  & \emptyset & \ydiagram{1} & \ydiagram{1,1}  &\ydiagram{2,1}  & \ydiagram{2,1,1}  & \ydiagram{3,2,1}  & \ydiagram{3,2,2} & \ydiagram{4,3,2,1} & \ydiagram[*(SkyBlue)]{4,4,3,1}  & \ydiagram[*(SkyBlue)]{4,4,3,2}  & \ydiagram[*(SkyBlue)]{4,4,4,3} & \ydiagram[*(SkyBlue)]{4,4,4,4,0} &  &  & &\\ %P8
& & &  &  &  &  & \scalebox{2.5}{$d_9$}  & \scalebox{2.5}{$d_{10}$}  & \emptyset  & \ydiagram{1} &  \ydiagram{2} &  \ydiagram{2,1} & \ydiagram{3,2}  & \ydiagram{3,2,1} & \ydiagram{4,3,1,1} & \ydiagram[*(SkyBlue)]{4,4,2,1}  & \ydiagram[*(SkyBlue)]{4,4,2,2}  & \ydiagram[*(SkyBlue)]{4,4,3,2} & \ydiagram[*(SkyBlue)]{4,4,4,3} & \ydiagram[*(SkyBlue)]{4,4,4,4,0}  &  & &\\ %P9
& & &  &  &  &  &  & \scalebox{2.5}{$d_{11}$}  & \scalebox{2.5}{$d_{12}$}  & \emptyset  & \ydiagram{1}  & \ydiagram{1,1}  & \ydiagram{2,1}  & \ydiagram{2,1,1} & \ydiagram{3,2,1,1} & \ydiagram{4,3,2,1}  & \ydiagram{4,3,2,2}  & \ydiagram{4,3,3,2} & \ydiagram[*(SkyBlue)]{4,4,3,3} &  \ydiagram[*(SkyBlue)]{4,4,4,3} & \ydiagram[*(SkyBlue)]{4,4,4,4,0}  & &\\ %P10
& & &  &  &  &  &  &  & \scalebox{2.5}{$d_{13}$}  & \scalebox{2.5}{$d_{14}$}  & \emptyset  & \ydiagram{1}  & \ydiagram{2}  & \ydiagram{2,1} & \ydiagram{3,2,1} & \ydiagram{4,3,2}  & \ydiagram{4,3,2,1}  & \ydiagram{4,3,3,1} & \ydiagram[*(SkyBlue)]{4,4,3,2} & \ydiagram[*(SkyBlue)]{4,4,4,2}  & \ydiagram[*(SkyBlue)]{4,4,4,3}  & \ydiagram[*(SkyBlue)]{4,4,4,4,0} &\\ %P11
& & &  &  &  &  &  &   &  & \scalebox{2.5}{$d_{15}$}  & \scalebox{2.5}{$d_{16}$}  & \scalebox{2.5}{$d_{17}$}  & \scalebox{2.5}{$d_{18}$}  & \scalebox{2.5}{$d_{19}$} & \scalebox{2.5}{$d_{20}$} & \scalebox{2.5}{$d_{21}$}  & \scalebox{2.5}{$d_{22}$}  & \scalebox{2.5}{$d_{23}$} & &  &  & &\\
%& & &  &  &  &  &  &  &  &  &  &  &  & & &  &  & & &  &  & &\\
%& & &  &  &  &  &  &  &  &  &  &  &  & & &  &  & & &  &  & &\\
\end{array}$}
\] \ytableausetup{boxsize=normal}
\[ \phantom{35246427572435135234544523456222} \ddots \]
where we have also labeled $23$ consecutive diagonals.
We then plot the following lattice path.
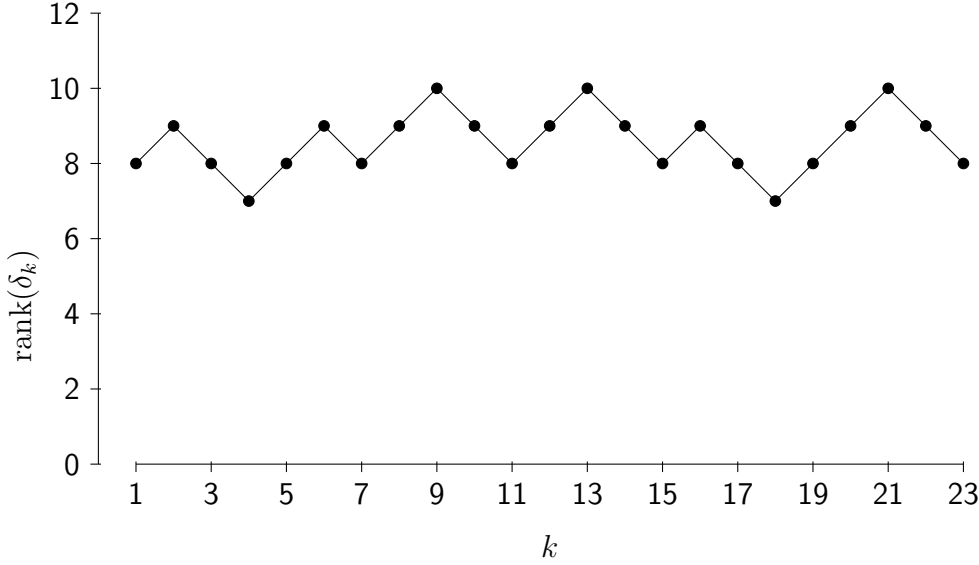
\[
\begin{tikzpicture}[y=.5cm, x=.5cm,font=\sffamily]
 	%axis
	\draw (1,0) -- coordinate (x axis mid) (23,0);
    	\draw (0,0) -- coordinate (y axis mid) (0,12);
    	%ticks
    	\foreach \x in {1,3,...,23}
     		\draw (\x,1pt) -- (\x,-3pt)
			node[anchor=north] {\x};
    	\foreach \y in {0,2,...,12}
     		\draw (1pt,\y) -- (-3pt,\y) 
     			node[anchor=east] {\y}; 
	%labels      
	\node[below=0.8cm] at (x axis mid) {$k$};
	\node[rotate=90, above=2cm, left=1cm] at (y axis mid) {$\rank(\delta_k)$};

	%plots
	\draw plot[mark=*, mark options={fill=black}] 
		file {LP.data};

\end{tikzpicture}
\]
Now notice that the number of right-hand endpoints of $\backslash$ steps equals the number of left-hand endpoints of $/$ steps at each height, the final key observation in the proof of Proposition~\ref{prop:dist}.
\end{example}

\subsection{Proof of Theorem~\ref{thm:homomesy}}
Consider $\bbb \in \Frame(m \times n)$ and let $T \in \inc^q(m \times n)$.
For $\mathcal{O}$ the $K$-promotion orbit of $T$, we have by Theorem~\ref{thm:framesame} that 
\[
\frac{\sum_{U \in \mathcal{O}} \wt_{\{ \bbb \}}(U)}{|\mathcal{O}|} = \frac{\sum_{i=0}^{q-1} \wt_{\{\bbb\}}(\pro^i(T))}{q}.
\]
By Proposition~\ref{prop:dist},
\[
\frac{\sum_{i=0}^{q-1} \wt_{\{\bbb\}}(\pro^i(T))}{q} = \frac{\sum_{i=0}^{q-1} \wt_{\{\bbb\}}(\pro^i(\evac(T)))}{q}.
\]
But by Theorem~\ref{thm:framesame} and Lemma~\ref{lemma:GD_facts}(c),
\[
\frac{\sum_{i=0}^{q-1} \wt_{\{\bbb\}}(\pro^i(\evac(T)))}{q} = \frac{\sum_{i=0}^{q-1} \wt_{\{\bbb\}}(\pro^{-i}(\evac(T)))}{q}   =  \frac{\sum_{i=0}^{q-1} \wt_{\{\bbb\}}(\evac(\pro^i(T)))}{q}.
\]
Finally by Proposition~\ref{prop:frame}, 
\[
\frac{\sum_{i=0}^{q-1} \wt_{\{\bbb\}}(\evac(\pro^i(T)))}{q} = \frac{\sum_{i=0}^{q-1} \wt_{\{\bbb\}}(\rot(\pro^i(T)))}{q} = \frac{\sum_{i=0}^{q-1} \left( q+1 - \wt_{\{\bbb^*\}}(\pro^i(T)) \right)}{q},
\]
where $\bbb^*$ is the image of $\bbb$ under rotating $m \times n$ by $180^\degree$. 

Hence putting these facts together, we have
\begin{align*}
\frac{\sum_{U \in \mathcal{O}} \wt_{\{\bbb, \bbb^*\}}(U)}{|\mathcal{O}|} &= \frac{\sum_{i=0}^{q-1} \left( q+1 - \wt_{\{\bbb^*\}}(\pro^i(T)) \right)}{q} + \frac{\sum_{i=0}^{q-1} \wt_{\{\bbb^*\}}(\pro^i(T))}{q} \\ 
&= \frac{\sum_{i=0}^{q-1} (q+1)}{q} = q+1.
\end{align*}

Thus for  $S$ any set of boxes in $\Frame(m \times n)$ that is fixed under $180^\degree$ rotation, we have
\[
\frac{\sum_{U \in \mathcal{O}} \wt_{S}(U)}{|\mathcal{O}|} = \frac{(q+1)|S|}{2}, 
\]
as desired. \qed

%\section*{Acknowledgements}
%Thank you all!

%%%%%%%%%%%%%%%%%%%%%%%%%%%%%%%%%%%%%%%%%%%%%%%%%%%%%%%%%%%%
%
%  Bibliography
%
%%%%%%%%%%%%%%%%%%%%%%%%%%%%%%%%%%%%%%%%%%%%%%%%%%%%%%%%%%%%

\bibliographystyle{amsalpha} 
\bibliography{frames.bib}

\end{document}